\documentclass[letterpaper,10pt,conference]{ieeeconf}  
\IEEEoverridecommandlockouts 
\usepackage{amssymb,amsmath,mathptmx}
\usepackage[pdftex]{graphicx}
  \graphicspath{{./pics/}}
  \DeclareGraphicsExtensions{.pdf}
\usepackage[caption=false,font=footnotesize]{subfig}
\usepackage{tikz}
  \usetikzlibrary{arrows,automata,shapes}
  \usetikzlibrary{decorations.pathmorphing}
  \usetikzlibrary{positioning}
  \usetikzlibrary{decorations.markings}
  \usetikzlibrary{intersections}

\newtheorem{theorem}{Theorem}
\newtheorem{lem}[theorem]{Lemma}
\newtheorem{cor}[theorem]{Corollary}
\newtheorem{proposition}[theorem]{Proposition}
\newtheorem{definition}[theorem]{Definition}
\newtheorem{prob}[theorem]{Problem}
\newtheorem{example}[theorem]{Example}

\newcommand{\eps}{\varepsilon}
\newcommand{\supcn}{\mbox{$\sup {\rm CN}$}}
\newcommand{\supCN}{\mbox{$\sup {\rm CN}$}}
\newcommand{\supccn}{\mbox{$\sup {\rm cCN}$}}

\newcommand{\supC}{\mbox{$\sup {\rm C}$}}
\newcommand{\supCC}{\mbox{$\sup {\rm cC}$}}

\title{\LARGE \bf Maximally Permissive Coordination Supervisory Control -- Towards Necessary and Sufficient Conditions}
\author{Jan Komenda, Tom{\' a}{\v s} Masopust, and Jan H. van Schuppen
  \thanks{J. Komenda and T. Masopust are affiliated with the Institute of Mathematics,
          Academy of Sciences of the Czech Republic,
          {\v Z}i{\v z}kova 22, 616 62 Brno, Czech Republic.
          T. Masopust is also affiliated with the Faculty of Computer Science, TU Dresden, Germany.
          {\tt\small komenda@math.cas.cz, masopust@math.cas.cz}
  }%
  \thanks{J. H. van Schuppen is affiliated with
          Van Schuppen Control Research,
          Gouden Leeuw 143, 1103 KB, Amsterdam, The Netherlands. \newline
          {\tt\small jan.h.van.schuppen@xs4all.nl}
  }%
}

\begin{document}

\maketitle
\thispagestyle{empty}
\pagestyle{empty}

%%%%%%%%%%%%%%%%%%%%%%%%%%%%%%%%%%%%%%%%%%%%%%%%%%%%%%%%%%%%%%%%%%%%%%%%%%%%%%%%
\begin{abstract}
  In this paper, we further develop the coordination control framework for discrete-event systems with both complete and partial observation. A new weaker sufficient condition for the computation of the supremal conditionally controllable sublanguage is presented. This result is then used for the computation of the supremal conditionally controllable and conditionally normal sublanguage. The paper further generalizes the previous study by considering general, non-prefix-closed languages.
\end{abstract} 

%%%%%%%%%%%%%%%%%%%%%%%%%%%%%%%%%%%%%%%%%%%%%%%%%%%%%%%%%%%%%%%%%%%%%%%%%%%%%%%%
\section{Introduction}\label{intro}
  Large scale discrete-event systems (DES) are often formed in a compositional way as a synchronous or asynchronous composition of smaller components, typically automata (or 1-safe Petri nets that can be viewed as products of automata). Supervisory control theory was proposed in~\cite{RW89} for automata as a formal approach that aims to solve the safety issue and nonblockingness. 
  
  A major issue is the computational complexity of the centralized supervisory control design, because the global system has an exponential number of states in the number of components. Therefore, a modular supervisory control of DES based on a compositional (local) control synthesis has been introduced and developed by many authors. Structural conditions have been derived for the local control synthesis to equal the global control synthesis in the case of both local and global specification languages. 
  
  Specifications are mostly defined over the global alphabet, which means that the global specifications are more relevant than the local specifications. However, several restrictive conditions have to be imposed on the modular plant such as mutual controllability (and normality) of local plant languages for maximal permissiveness of modular control, and other conditions are required for nonblockingness.
 
  For that reason, a coordination control approach was proposed for modular DES in~\cite{KvS08} and further developed in~\cite{automatica2011}. Coordination control can be seen as a reasonable trade-off between a purely modular control synthesis, which is in some cases unrealistic, and a global control synthesis, which is naturally prohibitive for high complexity reasons. The concept of a coordinator is useful for both safety and nonblockingness. The complete supervisor then consists of the coordinator, its supervisor, and the local supervisors for the subsystems. In~\cite{KvS08}, necessary and sufficient conditions are formulated for nonblockingness and safety, and a sufficient condition is formulated for the maximally permissive control synthesis satisfying a global specification using a coordinator. Later, in~\cite{automatica2011}, a procedure for a distributive computation of the supremal conditionally controllable sublanguage of a given specification has been proposed. We have extended coordination control for non-prefix-closed specification languages in~\cite{JDEDS} and for partial observations in~\cite{SCL}. 

  In this paper, we first propose a new sufficient condition for a distributive computation of the supremal conditionally controllable sublanguages. We show that it generalizes (is weaker than) both conditions we have introduced earlier in~\cite{JDEDS} and~\cite{automatica2011}. Then we revise (simplify) the concepts of conditional observability and conditional normality and present new sufficient conditions for a distributive computation of the supremal conditionally controllable and conditionally normal sublanguage.

  The paper is organized as follows. The next section recalls the basic concepts from the algebraic language theory that are needed in this paper. Our coordination control framework is briefly recalled in Section~\ref{sec:concepts}. In Section~\ref{sec:cc}, new results in coordination control with complete observations are presented: a new, weaker, sufficient condition for distributed computation of supremal conditionally controllable sublanguages. Section~\ref{sec:co} is dedicated to coordination control with partial observations, where the main concepts are simplified. Concluding remarks are in Section~\ref{conclusion}.

%%%%%%%%%%%%%%%%%%%%%%%%%%%%%%%%%%%%%%%%%%%%%%%%%%%%%%%%%%%%%%%%%%%%%%%%%%%%%%%%
\section{Preliminaries}\label{preliminaries}
  We now briefly recall the elements of supervisory control theory. The reader is referred to~\cite{CL08} for more details. Let $\Sigma$ be a finite nonempty set of {\em events}, and let $\Sigma^*$ denote the set of all finite words (strings) over $\Sigma$. The {\em empty word\/} is denoted by $\eps$. Let $|\Sigma|$ denote the cardinality of $\Sigma$.

  A {\em generator\/} is a quintuple $G=(Q,\Sigma, f, q_0, Q_m)$, where $Q$ is a finite nonempty set of {\em states}, $\Sigma$ is an {\em event set\/}, $f: Q \times \Sigma \to Q$ is a {\em partial transition function}, $q_0 \in Q$ is the {\em initial state}, and $Q_m\subseteq Q$ is the set of {\em marked states}. In the usual way, the transition function $f$ can be extended to the domain $Q \times \Sigma^*$ by induction. The behavior of $G$ is described in terms of languages. The language {\em generated\/} by $G$ is the set $L(G) = \{s\in \Sigma^* \mid f(q_0,s)\in Q\}$ and the language {\em marked\/} by $G$ is the set $L_m(G) = \{s\in \Sigma^* \mid f(q_0,s)\in Q_m\}\subseteq L(G)$.

  A {\em (regular) language\/} $L$ over an event set $\Sigma$ is a set $L\subseteq \Sigma^*$ such that there exists a generator $G$ with $L_m(G)=L$. The prefix closure of $L$ is the set $\overline{L}=\{w\in \Sigma^* \mid \text{there exists } u \in\Sigma^* \text{ such that } wu\in L\}$; $L$ is {\em prefix-closed\/} if $L=\overline{L}$.

  A {\em (natural) projection} $P: \Sigma^* \to \Sigma_o^*$, for some $\Sigma_o\subseteq \Sigma$, is a homomorphism defined so that $P(a)=\eps$, for $a\in \Sigma\setminus \Sigma_o$, and $P(a)=a$, for $a\in \Sigma_o$. The {\em inverse image} of $P$, denoted by $P^{-1} : \Sigma_o^* \to 2^{\Sigma^*}$, is defined as $P^{-1}(s)=\{w\in \Sigma^* \mid P(w) = s\}$. The definitions can naturally be extended to languages. The projection of a generator $G$ is a generator $P(G)$ whose behavior satisfies $L(P(G))=P(L(G))$ and $L_m(P(G))=P(L_m(G))$.

  A {\em controlled generator\/} is a structure $(G,\Sigma_c,P,\Gamma)$, where $G$ is a generator over $\Sigma$, $\Sigma_c \subseteq \Sigma$ is the set of {\em controllable events}, $\Sigma_{u} = \Sigma \setminus \Sigma_c$ is the set of {\em uncontrollable events}, $P:\Sigma^*\to \Sigma_o^*$ is the projection, and $\Gamma = \{\gamma \subseteq \Sigma \mid \Sigma_{u}\subseteq\gamma\}$ is the {\em set of control patterns}. A {\em supervisor\/} for the controlled generator $(G,\Sigma_c,P,\Gamma)$ is a map $S:P(L(G)) \to \Gamma$. A {\em closed-loop system\/} associated with the controlled generator $(G,\Sigma_c,P,\Gamma)$ and the supervisor $S$ is defined as the smallest language $L(S/G) \subseteq \Sigma^*$ such that (i) $\eps \in L(S/G)$ and (ii) if $s \in L(S/G)$, $sa\in L(G)$, and $a \in S(P(s))$, then $sa \in L(S/G)$. The marked behavior of the closed-loop system is defined as $L_m(S/G)=L(S/G)\cap L_m(G)$.
  
  Let $G$ be a generator over $\Sigma$, and let $K\subseteq L_m(G)$ be a specification. The aim of supervisory control theory is to find a nonblocking supervisor $S$ such that $L_m(S/G)=K$. The nonblockingness means that $\overline{L_m(S/G)} = L(S/G)$, hence $L(S/G)=\overline{K}$. It is known that such a supervisor exists if and only if $K$ is (i) {\em controllable\/} with respect to $L(G)$ and $\Sigma_u$, that is $\overline{K}\Sigma_u\cap L\subseteq \overline{K}$, (ii) {\em $L_m(G)$-closed}, that is $K = \overline{K}\cap L_m(G)$, and (iii) {\em observable\/} with respect to $L(G)$, $\Sigma_o$, and $\Sigma_c$, that is for all $s\in \overline{K}$ and $\sigma \in \Sigma_c$, $(s\sigma \notin \overline{K})$ and $(s\sigma \in L(G))$ imply that $P^{-1}[P(s)]\sigma \cap \overline{K} = \emptyset$, where $P:\Sigma^*\to \Sigma_o^*$, cf.~\cite{CL08}. 
  
  The synchronous product (parallel composition) of languages $L_1\subseteq \Sigma_1^*$ and $L_2\subseteq \Sigma_2^*$ is defined by $L_1\| L_2=P_1^{-1}(L_1) \cap P_2^{-1}(L_2) \subseteq \Sigma^*$, where $P_i: \Sigma^*\to \Sigma_i^*$, for $i=1,2$, are projections to local event sets. In terms of generators, see~\cite{CL08} for more details, it is known that $L(G_1 \| G_2) = L(G_1) \| L(G_2)$ and $L_m(G_1 \| G_2)= L_m(G_1) \| L_m(G_2)$.

%%%%%%%%%%%%%%%%%%%%%%%%%%%%%%%%%%%%%%%%%%%%%%%%%%%%%%%%%%%%%%%%%%%%%%%%%%%%%%%%
\section{Coordination Control Framework}\label{sec:concepts}
  A language $K\subseteq (\Sigma_1\cup \Sigma_2)^*$ is {\em conditionally decomposable\/} with respect to event sets $\Sigma_1$, $\Sigma_2$, and $\Sigma_k$, where $\Sigma_1\cap \Sigma_2\subseteq \Sigma_k$, if $K = P_{1+k} (K)\parallel P_{2+k} (K)$, where $P_{i+k}: (\Sigma_1\cup \Sigma_2)^*\to (\Sigma_i\cup \Sigma_k)^*$ is a projection, for $i=1,2$. Note that $\Sigma_k$ can always be extended so that the language $K$ becomes conditionally decomposable. A polynomial algorithm how to compute an extension can be found in~\cite{SCL12}. However, to find the minimal extension is NP-hard~\cite{JDEDS}.

  Now we recall the coordination control problem that is further developed in this paper.
  \begin{prob}\label{problem:controlsynthesis}
    Consider generators $G_1$ and $G_2$ over $\Sigma_1$ and $\Sigma_2$, respectively, and a generator $G_k$ (called a {\em coordinator\/}) over $\Sigma_k$ with $\Sigma_1\cap\Sigma_2\subseteq \Sigma_k$. Assume that a specification $K \subseteq L_m(G_1 \| G_2 \| G_k)$ and its prefix-closure $\overline{K}$ are conditionally decomposable with respect to event sets $\Sigma_1$, $\Sigma_2$, and $\Sigma_k$. The aim of coordination control is to determine nonblocking supervisors $S_1$, $S_2$, and $S_k$ for respective generators such that 
    \begin{align*}
      L_m(S_k/G_k)\subseteq P_k(K) && \& && L_m(S_i/ [G_i \| (S_k/G_k) ])\subseteq P_{i+k}(K)\,,
    \end{align*}
    for $i=1,2$, and
    \begin{flalign*}
      && L_m(S_1/ [G_1 \parallel (S_k/G_k) ]) \parallel L_m(S_2/ [G_2 \parallel (S_k/G_k) ]) & = K\,. && \hfill\diamond
    \end{flalign*}
  \end{prob}
  
  \medskip
  Recall that one way how to construct a coordinator is to set $G_k=P_k(G_1)\parallel P_k(G_2)$, cf.~\cite{automatica2011,JDEDS}.

%%%%%%%%%%%%%%%%%%%%%%%%%%%%%%%%%%%%%%%%%%%%%%%%%%%%%%%%%%%%%%%%%%%%%%%%%%%%%%%%
\section{Coordination Control with Complete Observations}\label{sec:cc}
  Conditional controllability introduced in~\cite{KvS08} and further studied in~\cite{ifacwc2011,SCL,automatica2011,JDEDS} plays the central role in coordination control. In what follows, we use the notation $\Sigma_{i,u}=\Sigma_i\cap \Sigma_u$ to denote the set of uncontrollable events of the event set $\Sigma_i$.

  \begin{definition}[Conditional controllability]\label{def:conditionalcontrollability}
    Let $G_1$ and $G_2$ be generators over $\Sigma_1$ and $\Sigma_2$, respectively, and let $G_k$ be a coordinator over $\Sigma_k$. A language $K\subseteq L_m(G_1\| G_2\| G_k)$ is {\em conditionally controllable\/} with respect to generators $G_1$, $G_2$, $G_k$ and uncontrollable event sets $\Sigma_{1,u}$, $\Sigma_{2,u}$, $\Sigma_{k,u}$ if
    \begin{enumerate}
      \item $P_k(K)$ is controllable with respect to $L(G_k)$ and $\Sigma_{k,u}$,
      \item $P_{1+k}(K)$ is controllable with respect to $L(G_1) \parallel \overline{P_k(K)}$ and $\Sigma_{1+k,u}$,
      \item $P_{2+k}(K)$ is controllable with respect to $L(G_2) \parallel \overline{P_k(K)}$ and $\Sigma_{2+k,u}$,
    \end{enumerate}
    where $\Sigma_{i+k,u}=(\Sigma_i\cup \Sigma_k)\cap \Sigma_u$, for $i=1,2$.
    $\hfill\triangleleft$
  \end{definition}

  The supremal conditionally controllable sublanguage always exists and equals to the union of all conditionally controllable sublanguages~\cite{JDEDS}. Let 
  \[
    \supCC(K, L, (\Sigma_{1,u}, \Sigma_{2,u}, \Sigma_{k,u}))
  \]
  denote the supremal conditionally controllable sublanguage of $K$ with respect to $L=L(G_1\|G_2\|G_k)$ and sets of uncontrollable events $\Sigma_{1,u}$, $\Sigma_{2,u}$, $\Sigma_{k,u}$. 
  
  The problem is now reduced to determining how to calculate the supremal conditionally-controllable sublanguage.
  
  Consider the setting of Problem~\ref{problem:controlsynthesis} and define the languages
  \begin{equation}\label{eqCC}
    \boxed{
    \begin{aligned}
      \supC_k     & =  \supC(P_k(K), L(G_k), \Sigma_{k,u})\\
      \supC_{1+k} & =  \supC(P_{1+k}(K), L(G_1) \parallel \overline{\supC_k}, \Sigma_{1+k,u})\\
      \supC_{2+k} & =  \supC(P_{2+k}(K), L(G_2) \parallel \overline{\supC_k}, \Sigma_{2+k,u})
    \end{aligned}}
  \end{equation}
  where $\supC(K,L,\Sigma_u)$ denotes the supremal controllable sublanguage of $K$ with respect to $L$ and $\Sigma_u$, see~\cite{CL08} for more details and algorithms.

  We have shown that $P_k(\supC_{i+k})\subseteq \supC_k$ always holds, for $i=1,2$, and that if the converse inclusion holds, we can compute the supremal conditionally-controllable sublanguage in a distributed way.
  \begin{theorem}[\cite{JDEDS}]
    Consider the setting of Problem~\ref{problem:controlsynthesis} and languages defined in~(\ref{eqCC}). If $\supC_k \subseteq P_k(\supC_{i+k})$, for $i=1,2$, then 
    \[
      \supC_{1+k} \parallel \supC_{2+k} = \supCC(K, L, (\Sigma_{1,u}, \Sigma_{2,u}, \Sigma_{k,u}))\,,
    \]
    where $L=L(G_1\|G_2\|G_k)$. 
    \hfill\QED
  \end{theorem}

  We can now further improve this result by introducing a weaker condition for nonconflicting supervisors. Recall that two languages $L_1$ and $L_2$ are {\em nonconflicting\/} if $\overline{L_1\|L_2} = \overline{L_1} \| \overline{L_2}$.
  
  \begin{theorem}\label{thmNEW}
    Consider the setting of Problem~\ref{problem:controlsynthesis} and languages defined in~(\ref{eqCC}). Assume that $\supC_{1+k}$ and $\supC_{2+k}$ are nonconflicting. If $P_k(\supC_{1+k}) \cap P_k(\supC_{2+k})$ is controllable with respect to $L(G_k)$ and $\Sigma_{k,u}$, then 
    \[
      \supC_{1+k} \parallel \supC_{2+k} = \supCC(K, L, (\Sigma_{1,u}, \Sigma_{2,u}, \Sigma_{k,u}))\,, 
    \]
    where $L=L(G_1\|G_2\|G_k)$.
  \end{theorem}
  \begin{proof}
    Let $\supCC = \supCC(K, L, (\Sigma_{1,u}, \Sigma_{2,u}, \Sigma_{k,u}))$ and $M = \supC_{1+k} \parallel \supC_{2+k}$. To prove $M \subseteq \supCC$, we show that $M \subseteq P_{1+k}(K) \parallel P_{2+k}(K) = K$ (by conditional decomposability) is conditionally controllable with respect to $G_1, G_2, G_k$ and $\Sigma_{1,u}, \Sigma_{2,u}, \Sigma_{k,u}$. However, $P_k(M) = P_k(\supC_{1+k}) \cap P_k(\supC_{2+k})$ (by Lemma~\ref{lemma:Wonham}) is controllable with respect to $L(G_k)$ and $\Sigma_{k,u}$ by the assumption. Furthermore, $P_{1+k}(M) = \supC_{1+k} \parallel P^{2+k}_k(\supC_{2+k})$ implies that $\supC_{1+k} \parallel P^{1+k}_k(\supC_{1+k}) \parallel P^{2+k}_k(\supC_{2+k}) = \supC_{1+k} \parallel P^{2+k}_k(\supC_{2+k}) = P_{1+k}(M)$. Thus, $P_{1+k}(M)=\supC_{1+k} \parallel [P^{1+k}_k(\supC_{1+k}) \parallel P^{2+k}_k(\supC_{2+k})]$ is controllable with respect to $[L(G_1)\|\overline{\supC_k}]\parallel \overline{P_k(M)} = L(G_1)\parallel \overline{P_k(M)}$ by Lemma~\ref{feng} (because nonconflictingness of $\supC_{1+k}$ and $\supC_{2+k}$ implies nonconflictingness of $\supC_{1+k}$ and $P^{1+k}_k(\supC_{1+k}) \parallel P^{2+k}_k(\supC_{2+k})$) and by the fact that $P^{i+k}_k(\supC_{i+k})\subseteq \supC_k$, for $i=1,2$, cf.~\cite{JDEDS}. Similarly for $P_{2+k}(M)$, hence $M \subseteq \supCC$.

    To prove the opposite inclusion, it is sufficient to show by Lemma~\ref{lem11} that $P_{i+k}(\supCC) \subseteq \supC_{i+k}$, for $i=1,2$. To prove this note that $P_{1+k}(\supCC)$ is controllable with respect to $L(G_1)\parallel \overline{P_k(\supCC)}$ and $\Sigma_{1+k,u}$, and $L(G_1) \parallel \overline{P_k(\supCC)}$ is controllable with respect to $L(G_1) \parallel \overline{\supC_k}$ and $\Sigma_{1+k,u}$ (by Lemma~\ref{feng}) because $P_k(\supCC)$ being controllable with respect to $L(G_k)$ is also controllable with respect to $\overline{\supC_k}\subseteq L(G_k)$. By the transitivity of controllability (Lemma~\ref{lem_transC}), $P_{1+k}(\supCC)$ is controllable with respect to $L(G_1) \parallel \overline{\supC_k}$ and $\Sigma_{1+k,u}$, which implies that $P_{1+k}(\supCC) \subseteq \supC_{1+k}$. The other case is analogous, hence $\supCC \subseteq M$ and the proof is complete.
  \end{proof}
  
  Note that the controllability condition of Theorem~\ref{thmNEW} is weaker than to require that $\supC_k \subseteq P_k(\supC_{i+k})$, for $i=1,2$.
  \begin{proposition}\label{weaker}
    If $\supC_k \subseteq P_k(\supC_{i+k})$, for $i=1,2$, then $P_k(\supC_{1+k}) \cap P_k(\supC_{2+k})$ is controllable with respect to $L(G_k)$ and $\Sigma_{k,u}$.
  \end{proposition}
  \begin{proof}
    This is obvious, because due to the converse inclusion being always true we have that $P_k(\supC_{i+k})=\supC_k$, for $i=1,2$. Hence, $P_k(\supC_{1+k}) \cap P_k(\supC_{2+k})=\supC_k$ is controllable with respect to $L(G_k)$ and $\Sigma_{k,u}$ by definition of $\supC_k$.
  \end{proof} 
  
  Using the example from~\cite{JDEDS} we can now show that there are languages such that $\supC_k\not\subseteq P_k(\supC_{i+k})$, but such that $P_k(\supC_{1+k}) \cap P_k(\supC_{2+k})$ is controllable with respect to $L(G_k)$ and $\Sigma_{k,u}$. 
  \begin{example}
    Let $G_1$ and $G_2$ be generators as shown in Fig.~\ref{figEx}, and let $K$ be the language of the generator shown in Fig.~\ref{figExx}. Let $\Sigma_c=\{a_1,a_2,c\}$ and $\Sigma_k=\{a_1,a_2,c,u\}$. Let the coordinator $G_k=P_k(G_1)\parallel P_k(G_2)$. Then $K$ is conditionally decomposable, $\supC_k=\overline{\{a_1a_2,a_2a_1\}}$, $\supC_{1+k}=\overline{\{a_2a_1u_1\}}$, $\supC_{2+k}=\overline{\{a_1a_2u_2\}}$, and $\supC_k \not\subseteq P_k(\supC_{i+k})$. However, $P_k(\supC_{1+k}) \cap P_k(\supC_{2+k})=\{\eps\}$ is controllable with respect to $L(G_k)$ and $\Sigma_{k,u}$. $\hfill\triangleleft$
  \end{example}
  \begin{figure}[htb]
    \centering
    \subfloat{
      \begin{tikzpicture}[->,>=stealth,shorten >=1pt,auto,node distance=1.5cm,
        state/.style={circle,minimum size=0mm, very thin,draw=black,initial text=}]
        \node[state,initial,accepting] (1) {1};
        \node[state,accepting] (2) [below right of=1] {2};
        \node[state,accepting] (3) [below left of=2] {3};
        \node[state,accepting] (4) [below left of=1] {4};
        \path
          (1) edge node {$a_1$} (2)
          (2) edge node {$u_1$} (3)
          (1) edge node[above left] {$c$} (4)
          (4) edge node[below left] {$u$} (3);
      \end{tikzpicture}
    }
    \qquad
    \subfloat{
      \begin{tikzpicture}[->,>=stealth,shorten >=1pt,auto,node distance=1.5cm,
        state/.style={circle,minimum size=0mm, very thin,draw=black,initial text=}]
        \node[state,initial,accepting] (1) {1};
        \node[state,accepting] (2) [below right of=1] {2};
        \node[state,accepting] (3) [below left of=2] {3};
        \node[state,accepting] (4) [below left of=1] {4};
        \path
          (1) edge node {$a_2$} (2)
          (2) edge node {$u_2$} (3)
          (1) edge node[above left] {$c$} (4)
          (4) edge node[below left] {$u$} (3);
      \end{tikzpicture}
    }
    \caption{Generators $G_1$ and $G_2$.}
    \label{figEx}
  \end{figure}
  \begin{figure}[htb]
    \centering
    \begin{tikzpicture}[->,>=stealth,shorten >=1pt,auto,node distance=1.5cm,
      state/.style={circle,minimum size=0mm, very thin,draw=black,initial text=}]
      \node[state,initial,accepting] (1) {1};
      \node[state,accepting] (2) [above right of=1] {2};
      \node[state,accepting] (3) [right of=2] {3};
      \node[state,accepting] (4) [right of=3] {4};
      \node[state,accepting] (5) [right of=1] {5};
      \node[state,accepting] (6) [right of=5] {6};
      \node[state,accepting] (7) [right of=6] {7};
      \path
        (1) edge node {$a_1$} (2)
        (2) edge node {$a_2$} (3)
        (3) edge node {$u_2$} (4)
        (1) edge node {$a_2$} (5)
        (5) edge node {$a_1$} (6)
        (6) edge node {$u_1$} (7);
    \end{tikzpicture}
    \caption{Specification $K$.}
    \label{figExx}
  \end{figure}

  On the other hand, $P_k(\supC_{1+k}) \cap P_k(\supC_{2+k})$ is not always controllable with respect to $L(G_k)$ and $\Sigma_{k,u}$.
  \begin{example}
    Let $G_1$ and $G_2$ be generators as shown in Fig.~\ref{figEx2}, and let $K$ be the language of the generator shown in Fig.~\ref{figExx2}. Let $\Sigma_c=\{a,c_1,c_2\}$ and $\Sigma_k=\{a,b\}$. Let the coordinator $G_k=P_k(G_1)\parallel P_k(G_2)$. Then the language $K$ is conditionally decomposable, $\supC_k=\overline{\{b\}}$, $\supC_{1+k}=\overline{\{c_1b\}}$, $\supC_{2+k}=\{\eps\}$, and $P_k(\supC_{1+k}) \cap P_k(\supC_{2+k})=\{\eps\}$ is not controllable with respect to $L(G_k)=\overline{\{ab,b\}}$ and $\Sigma_{k,u}=\{b\}$. \hfill$\triangleleft$
  \end{example}
  \begin{figure}[htb]
    \centering
    \subfloat{
      \begin{tikzpicture}[->,>=stealth,shorten >=1pt,auto,node distance=1.5cm,
        state/.style={circle,minimum size=0mm, very thin,draw=black,initial text=}]
        \node[state,initial,accepting] (1) {1};
        \node[state,accepting] (2) [below right of=1] {2};
        \node[state,accepting] (3) [below left of=2] {3};
        \node[state,accepting] (4) [below left of=1] {4};
        \path
          (1) edge node {$c_1$} (2)
          (2) edge node {$b$} (3)
          (1) edge node[above left] {$a$} (4)
          (4) edge node[below left] {$b$} (3);
      \end{tikzpicture}
    }
    \qquad
    \subfloat{
      \begin{tikzpicture}[->,>=stealth,shorten >=1pt,auto,node distance=1.5cm,
        state/.style={circle,minimum size=0mm, very thin,draw=black,initial text=}]
        \node[state,initial,accepting] (1) {1};
        \node[state,accepting] (2) [below right of=1] {2};
        \node[state,accepting] (3) [below left of=2] {3};
        \node[state,accepting] (4) [below left of=1] {4};
        \path
          (1) edge node {$c_2$} (2)
          (2) edge node {$b,u_2$} (3)
          (1) edge node[above left] {$a$} (4)
          (4) edge node[below left] {$b$} (3);
      \end{tikzpicture}
    }
    \caption{Generators $G_1$ and $G_2$.}
    \label{figEx2}
  \end{figure}
  \begin{figure}[htb]
    \centering
    \begin{tikzpicture}[->,>=stealth,shorten >=1pt,auto,node distance=1.5cm,
      state/.style={circle,minimum size=0mm, very thin,draw=black,initial text=}]
      \node[state,initial,accepting] (1) {1};
      \node[state,accepting] (2) [above right of=1] {2};
      \node[state,accepting] (3) [right of=2] {3};
      \node[state,accepting] (4) [right of=3] {4};
      \node[state,accepting] (5) [right of=1] {5};
      \path
        (1) edge node {$c_1$} (2)
        (2) edge node {$c_2$} (3)
        (3) edge node {$b$} (4)
        (1) edge node {$c_2$} (5)
        (5) edge node {$c_1$} (3)
        (1) edge[bend right=45] node {$a$} (4);
    \end{tikzpicture}
    \caption{Specification $K$.}
    \label{figExx2}
  \end{figure}
  
  Recall that it is still an open problem how to compute the supremal conditionally-controllable sublanguage for a general, non-prefix-closed language. 
  
  The following conditions were required in~\cite{automatica2011} to prove the main result for prefix-closed languages. We recall the result here and show that the previous condition is a weaker condition than the one required in~\cite{automatica2011}. 
  
  The projection $P:\Sigma^* \to \Sigma_0^*$, where $\Sigma_0\subseteq \Sigma$, is an {\em $L$-observer} for $L\subseteq \Sigma^*$ if, for all $t\in P(L)$ and $s\in \overline{L}$, $P(s)$ is a prefix of $t$ implies that there exists $u\in \Sigma^*$ such that $su\in L$ and $P(su)=t$. 
  
  The projection $P:\Sigma^* \to \Sigma_0^*$ is {\em output control consistent\/} (OCC) for $L\subseteq\Sigma^*$ if for every $s\in \overline{L}$ of the form $s=\sigma_1\dots\sigma_\ell$ or $s = s'\sigma_0 \sigma_1 \dots \sigma_\ell$, $\ell\ge 1$, where $s'\in \Sigma^*$, $\sigma_0, \sigma_\ell\in \Sigma_k$, and $\sigma_i \in \Sigma\setminus \Sigma_k$, for $i=1,2,\dots,\ell-1$, if $\sigma_\ell \in \Sigma_{u}$, then $\sigma_i \in \Sigma_{u}$, for all $i=1,2,\dots,\ell-1$. 
  
  The OCC condition can be replaced by a weaker condition called local control consistency (LCC) discussed in~\cite{SB11,SB08}, see~\cite{JDEDS}. Let $L$ be a prefix-closed language over $\Sigma$, and let $\Sigma_0$ be a subset of $\Sigma$. The projection $P:\Sigma^*\to \Sigma_0^*$ is {\em locally control consistent\/} (LCC) with respect to a word $s\in L$ if for all events $\sigma_u\in \Sigma_0\cap \Sigma_u$ such that $P(s)\sigma_u\in P(L)$, it holds that either there does not exist any word $u\in (\Sigma\setminus \Sigma_0)^*$ such that $su\sigma_u \in L$, or there exists a word $u\in (\Sigma_u\setminus \Sigma_0)^*$ such that $su\sigma_u \in L$. The projection $P$ is LCC with respect to $L$ if $P$ is LCC for all words of $L$.

  \begin{theorem}[\cite{JDEDS}]
    Consider the setting of Problem~\ref{problem:controlsynthesis} with a prefix-closed specification $K$. Consider the languages defined in~(\ref{eqCC}) and assume that $\supC_{1+k}$ and $\supC_{2+k}$ are nonconflicting. Let $P^{i+k}_k$ be an $(P^{i+k}_i)^{-1}L(G_i)$-observer and OCC (resp. LCC) for $(P^{i+k}_i)^{-1}L(G_i)$, for $i=1,2$. Then 
    \[
      \supC_{1+k} \parallel \supC_{2+k} = \supCC(K, L, (\Sigma_{1,u}, \Sigma_{2,u}, \Sigma_{k,u}))\,, 
    \]
    where $L=L(G_1\|G_2\|G_k)$. \hfill\QED
  \end{theorem}

  We can now prove that the assumptions of the previous theorem are stronger than the assumptions of Theorem~\ref{thmNEW}. This is shown in the following lemma and corollary, and summarized in Theorem~\ref{thm10}.
  \begin{lem}\label{lem8}
    Consider the setting of Problem~\ref{problem:controlsynthesis} and the languages defined in~(\ref{eqCC}). Assume that $\supC_{1+k}$ and $\supC_{2+k}$ are nonconflicting, and let the projection $P^{i+k}_k:(\Sigma_i\cup\Sigma_k)^*\to\Sigma_k^*$ be an $(P^{i+k}_i)^{-1}L(G_i)$-observer and OCC (resp. LCC) for $(P^{i+k}_i)^{-1}L(G_i)$, for $i=1,2$. Then $P^{1+k}_k(\supC_{1+k})\cap P^{2+k}_k(\supC_{2+k})$ is controllable with respect to $P_k(L(G_1))\parallel P_k(L(G_2))\parallel L(G_k)$ and $\Sigma_{k,u}$.
  \end{lem}
  \begin{proof}
    Since $\Sigma_{1+k}\cap\Sigma_{2+k}=\Sigma_k$, Lemma~\ref{lemma:Wonham} implies that $P^{1+k}_k(\supC_{1+k})\cap P^{2+k}_k(\supC_{2+k}) = P_k(\supC_{1+k}\parallel \supC_{2+k})$. By Lemma~\ref{obsComposition}, because $P_k^k = id$ is an $L(G_k)$-observer, $P_k$ is an $L:=L(G_1\|G_2\|G_k)$-observer. Assume that $t\in \overline{P_k(\supC_{1+k}\parallel \supC_{2+k})}$, $u\in\Sigma_{k,u}$, and $tu\in P_k(L)=P_k(L(G_1))\parallel P_k(L(G_2))\parallel L(G_k)$. Then there exists $s\in \overline{\supC_{1+k}\parallel \supC_{2+k}}\subseteq L$ such that $P_k(s)=t$. By the observer property, there exists $v$ such that $sv\in L$ and $P_k(sv)=tu$, that is, $v=v_1u$ with $P_k(v_1u)=u$. By the OCC property, $v_1\in\Sigma_u^*$, and by controllability of $\supC_{i+k}$, $i=1,2$, $sv_1u\in \overline{\supC_{1+k}} \parallel \overline{\supC_{2+k}} = \overline{\supC_{1+k}\parallel \supC_{2+k}}$, hence $tu\in \overline{P_k(\supC_{1+k}\parallel \supC_{2+k})}$.
    
    Similarly for LCC: from $sv=sv_1u\in L$, by the LCC property, there exists $v_2\in (\Sigma_u\setminus\Sigma_k)^*$ such that $sv_2u\in L$, and by controllability of $\supC_{i+k}$, $i=1,2$, $sv_2u\in \overline{\supC_{1+k}} \parallel \overline{\supC_{2+k}} = \overline{\supC_{1+k}\parallel \supC_{2+k}}$, hence $tu\in \overline{P_k(\supC_{1+k}\parallel \supC_{2+k})}$.
  \end{proof}
  
  Note that if $L(G_k)\subseteq P_k(L(G_1))\parallel P_k(L(G_2))$, which is actually the way we usually define the coordinator (since we usually define $G_k=P_k(G_1)\parallel P_k(G_2)$), we get the following corollary.
  \begin{cor}
    Consider the setting of Problem~\ref{problem:controlsynthesis} with $L(G_k)\subseteq P_k(L(G_1))\parallel P_k(L(G_2))$ and the languages defined in~(\ref{eqCC}). Assume that $\supC_{1+k}$ and $\supC_{2+k}$ are nonconflicting. Let $P^{i+k}_k:(\Sigma_i\cup\Sigma_k)^*\to\Sigma_k^*$ be an $(P^{i+k}_i)^{-1}L(G_i)$-observer and OCC (resp. LCC) for $(P^{i+k}_i)^{-1}L(G_i)$, for $i=1,2$. Then $P^{1+k}_k(\supC_{1+k})\cap P^{2+k}_k(\supC_{2+k})$ is controllable with respect to $L(G_k)$ and $\Sigma_{k,u}$.
  \end{cor}
  \begin{proof}
    The assumption $L(G_k)\subseteq P_k(L(G_1))\parallel P_k(L(G_2))$ implies that $P_k(L(G_1))\parallel P_k(L(G_2))\parallel L(G_k) = L(G_k)$.
  \end{proof}
  
  Finally, as a consequence of Lemma~\ref{lem8} and Theorem~\ref{thmNEW}, we obtain the following result.
  \begin{theorem}\label{thm10}
    Consider the setting of Problem~\ref{problem:controlsynthesis} with $L(G_k)\subseteq P_k(L(G_1))\parallel P_k(L(G_2))$ and the languages defined in~(\ref{eqCC}). Assume that $\supC_{1+k}$ and $\supC_{2+k}$ are nonconflicting. Let $P^{i+k}_k$ be an $(P^{i+k}_i)^{-1}L(G_i)$-observer and OCC (resp. LCC) for $(P^{i+k}_i)^{-1}L(G_i)$, for $i=1,2$. Then 
    \[
      \supC_{1+k} \parallel \supC_{2+k} = \supCC(K, L, (\Sigma_{1,u}, \Sigma_{2,u}, \Sigma_{k,u}))\,, 
    \]
    where $L=L(G_1\|G_2\|G_k)$. \hfill\QED
  \end{theorem}

%%%%%%%%%%%%%%%%%%%%%%%%%%%%%%%%%%%%%%%%%%%%%%%%%%%%%%%%%%%%%%%%%%%%%%%%%%%%%%%%
\section{Coordination Control with Partial Observations}\label{sec:co}
  In this section, we study coordination control of modular DES, where both the coordinator supervisor and the local supervisors have incomplete (partial) information about occurrences of their events and, hence, they do not know the exact state of the coordinator and the local plants.

  The contribution of this section is twofold. First, basic concepts of conditional observability and conditional normality are simplified in a similar way as it has been done in~\cite{JDEDS}. Then, we propose new sufficient conditions for a distributed computation of the supremal conditionally normal and conditionally controllable sublanguage. In particular, a weaker condition is presented that combines the weaker condition for distributed computation of the supremal conditionally controllable sublanguage presented in Section~\ref{sec:cc} with a similar condition for computation of the supremal conditionally normal sublanguage. Furthermore, a stronger condition is presented that is easy to check and that works also for non-prefix-closed specifications.
% and revisit the conditions under which it coincides with the supremal monolithic solution, i.e.  the supremal controllable and normal sublanguage..
  
\subsection{Conditional Observability}
  For coordination control with partial observations, the notion of conditional observability is of the same importance as observability for monolithic supervisory control theory with partial observations.
  
  \begin{definition}\label{def:conditionalobservability}
    Let $G_1$ and $G_2$ be generators over $\Sigma_1$ and $\Sigma_2$, respectively, and let $G_k$ be a coordinator over $\Sigma_k$. A language $K\subseteq L_m(G_1\| G_2\| G_k)$ is {\em conditionally observable\/} with respect to generators $G_1$, $G_2$, $G_k$, controllable sets $\Sigma_{1,c}$, $\Sigma_{2,c}$, $\Sigma_{k,c}$, and projections $Q_{1+k}$, $Q_{2+k}$, $Q_{k}$, where $Q_i: \Sigma_i^*\to \Sigma_{i,o}^*$, for $i=1+k,2+k,k$, if
    \begin{enumerate}
      \item $P_k(K)$ is observable with respect to $L(G_k)$, $\Sigma_{k,c}$, $Q_{k}$,
      \item $P_{1+k}(K)$ is observable with respect to $L(G_1) \parallel \overline{P_k(K)}$, $\Sigma_{1+k,c}$, $Q_{1+k}$,
      \item $P_{2+k}(K)$ is observable with respect to $L(G_2) \parallel \overline{P_k(K)}$, $\Sigma_{2+k,c}$, $Q_{2+k}$, 
    \end{enumerate}
    where $\Sigma_{i+k,c}=\Sigma_c \cap (\Sigma_i \cup \Sigma_k)$, for $i=1,2$.
    $\hfill\triangleleft$
  \end{definition}
  
  Analogously to the notion of $L_m(G)$-closed languages, we recall the notion of conditionally-closed languages defined in~\cite{ifacwc2011}. A nonempty language $K$ over $\Sigma$ is {\em conditionally closed\/} with respect to generators $G_1$, $G_2$, $G_k$ if
  \begin{enumerate}
    \item\label{ccl1} $P_k(K)$ is $L_m(G_k)$-closed,
    \item\label{ccl2} $P_{1+k}(K)$ is $L_m(G_1) \parallel P_k(K)$-closed,
    \item\label{ccl3} $P_{2+k}(K)$ is $L_m(G_2) \parallel P_k(K)$-closed.
  \end{enumerate}
  
  We can now formulate the main result for coordination control with partial observation. This is a generalization of a similar result for prefix-closed languages given in~\cite{SCL} stated moreover with the above defined simplified (but equivalent) form of conditional observability.
  \begin{theorem}\label{thm1}
    Consider the setting of Problem~\ref{problem:controlsynthesis}. There exist nonblocking supervisors $S_1$, $S_2$, $S_k$ such that
    \begin{equation}\tag{1}\label{eq:controlsynthesissafety}
        L_m(S_1/[G_1 \| (S_k/G_k)]) \parallel L_m(S_2/[G_2 \| (S_k/G_k)]) = K
    \end{equation}
    if and only if $K$ is
    (i) conditionally controllable with respect generators $G_1$, $G_2$, $G_k$ and $\Sigma_{1,u}$, $\Sigma_{2,u}$, $\Sigma_{k,u}$, 
    (ii) conditionally closed with respect to generators $G_1$, $G_2$, $G_k$, and 
    (iii) conditionally observable with respect to $G_1$, $G_2$, $G_k$, event sets $\Sigma_{1,c}$, $\Sigma_{2,c}$, $\Sigma_{k,c}$, and projections $Q_{1+k}$, $Q_{2+k}$, $Q_{k}$ from $\Sigma_i^*$ to $\Sigma_{i,o}^*$, for $i=1+k,2+k,k$.
  \end{theorem}
  \begin{proof}
    (If) Since $K\subseteq L_m(G_1\| G_2\| G_k)$, we have $P_k(K) \subseteq L_m(G_k)$ is controllable with respect to $L(G_k)$ and $\Sigma_{k,u}$, $L_m(G_k)$-closed, and observable with respect to $L(G_k)$, $\Sigma_{k,c}$, and $Q_k$. It follows, see~\cite{CL08}, that there exists a nonblocking supervisor $S_k$ such that $L_m(S_k/G_k)=P_k(K)$. Similarly, we have $P_{1+k}(K) \subseteq L_m(G_1)\parallel L_m(G_k)$ and $P_{1+k}(K) \subseteq (P_{k}^{1+k})^{-1} P_k(K)$, hence $P_{1+k}(K) \subseteq L_m(G_1) \parallel L_m(G_k) \parallel P_k(K) =  L_m(G_1) \parallel P_k(K) =  L_m(G_1) \parallel L_m(S_k/G_k)$. This, together with the assumption that $K$ is conditionally controllable, conditionally closed, and conditionally observable imply, see~\cite{CL08}, that there exists a nonblocking supervisor $S_{1}$ such that $L_m(S_1/ [ G_1 \| (S_k/G_k)]) = P_{1+k}(K)$. A similar argument shows that there exists a nonblocking supervisor $S_2$ such that $L_m(S_2/ [ G_2 \| (S_k/G_k)]) = P_{2+k}(K)$. Since $K$ is conditionally decomposable, $L_m(S_1/[G_1 \| (S_k/G_k)]) \parallel L_m(S_2/[G_2 \| (S_k/G_k)]) =  P_{1+k}(K) \parallel P_{2+k}(K) = K$.

    (Only if) To prove this direction, projections $P_k$, $P_{1+k}$, $P_{2+k}$ are applied to~(\ref{eq:controlsynthesissafety}). The closed-loop languages can be written as synchronous products, thus (\ref{eq:controlsynthesissafety}) can be written as $K = L_m(S_1) \parallel L_m(G_1) \parallel L_m(S_k) \parallel L_m(G_k) \parallel L_m(S_2) \parallel L_m(G_2) \parallel L_m(S_k) \parallel L_m(G_k)$, which gives $P_k(K) \subseteq L_m(S_k) \parallel L_m(G_k) =  L_m(S_k/G_k)$. On the other hand, $L_m(S_k/G_k) \subseteq P_k(K)$, see Problem~\ref{problem:controlsynthesis}, hence $L_m(S_k/G_k) = P_k(K)$, which means, according to the basic theorem of supervisory control~\cite{CL08}, that $P_k(K)$ is controllable with respect to $L(G_k)$ and $\Sigma_{k,u}$, $L_m(G_k)$-closed, and observable with respect to $L(G_k)$, $\Sigma_{k,c}$, and $Q_k$. Now, the application of $P_{1+k}$ to (\ref{eq:controlsynthesissafety}) gives $P_{1+k}(K) \subseteq L_m(S_1/ [ G_1 \| (S_k/G_k)]) \subseteq P_{1+k}(K)$. According to the basic theorem of supervisory control, $P_{1+k}(K)$ is controllable with respect to $L(G_1 \| (S_k/G_k))$ and $\Sigma_{1+k,u}$, $L_m(G_1\|(S_k/G_k))$-closed, and observable with respect to $L(G_1\| (S_k/G_k))$, $\Sigma_{1+k,c}$, and $Q_{1+k}$. Similarly, $P_{2+k}(K)$ is controllable with respect to $L(G_2 \| (S_k/G_k))$ and $\Sigma_{2+k,u}$, $L_m(G_2\|(S_k/G_k))$-closed, and observable with respect to $L(G_2\| (S_k/G_k))$, $\Sigma_{2+k,c}$, and $Q_{2+k}$, which was to be shown.
  \end{proof}

%%%%%%%%%%%%%%%%%%%%%%%%%%%%%%%%%%%%%%%%%%%%%%%%%%%%%%%%%%%%%%%%%%%%%%%%%%%%%%%%
\subsection{Conditional normality}\label{sec:cn}
  It is well known that supremal observable sublanguages do not exist in general and it is also the case of conditionally observable sublanguages. Therefore, a stronger concept of language normality has been introduced.
  
  Let $G$ be a generator over $\Sigma$, and let $P:\Sigma^* \to \Sigma_o^*$ be a projection. A language $K\subseteq L_m(G)$ is {\em normal\/} with respect to $L(G)$ and $P$ if $\overline{K} = P^{-1}P(\overline{K})\cap L(G)$. It is known that normality implies observability~\cite{CL08}.
  
  \begin{definition}\label{def:conditionalnormality}
    Let $G_1$ and $G_2$ be generators over $\Sigma_1$ and $\Sigma_2$, respectively, and let $G_k$ be a coordinator over $\Sigma_k$. A language $K\subseteq L_m(G_1\| G_2\| G_k)$ is {\em conditionally normal\/} with respect to generators $G_1, G_2, G_k$ and projections $Q_{1+k}, Q_{2+k}$, $Q_{k}$, where $Q_i: \Sigma_i^*\to \Sigma_{i,o}^*$, for $i=1+k,2+k,k$, if
    \begin{enumerate}
      \item $P_k(K)$ is normal with respect to $L(G_k)$ and $Q_{k}$,
      \item $P_{1+k}(K)$ is normal with respect to $L(G_1) \parallel \overline{P_k(K)}$ and $Q_{1+k}$,
      \item $P_{2+k}(K)$ is normal with respect to $L(G_2) \parallel \overline{P_k(K)}$ and $Q_{2+k}$. $\hfill\triangleleft$
    \end{enumerate}
  \end{definition}
  
  The following result is an immediate application of conditional normality in coordination control.
  \begin{theorem}\label{thm0}
    Consider the setting of Problem~\ref{problem:controlsynthesis}. If the specification $K$ is conditionally controllable with respect to $G_1, G_2, G_k$ and $\Sigma_{1,u}, \Sigma_{2,u}, \Sigma_{k,u}$, conditionally closed with respect to $G_1, G_2$, $G_k$, and conditionally normal with respect to $G_1, G_2, G_k$ and projections $Q_{1+k}, Q_{2+k}, Q_{k}$ from $\Sigma_i^*$ to $\Sigma_{i,o}^*$, for $i=1+k,2+k,k$, then there exist nonblocking supervisors $S_1$, $S_2$, $S_k$ such that 
    \[
      L_m(S_1/[G_1 \| (S_k/G_k)]) \parallel L_m(S_2/[G_2 \| (S_k/G_k)]) = K\,.
    \]
  \end{theorem}
  \medskip
  \begin{proof}
    As normality implies observability, the proof follows immediately from Theorem~\ref{thm1}.
  \end{proof}
  
  The following result was proved for prefix-closed languages in~\cite{SCL}. Here we generalize it for not necessarily prefix-closed languages.
  \begin{theorem}\label{existence2}
    The supremal conditionally normal sublanguage always exists and equals to the union of all conditionally normal sublanguages.
  \end{theorem}
  \begin{proof}
    We show that conditional normality is preserved under union. Let $I$ be an index set, and let $K_i$ be conditionally normal sublanguages of $K\subseteq L_m(G_1\|G_2\|G_k)$ with respect to generators $G_1$, $G_2$, $G_k$ and projections $Q_{1+k}$, $Q_{2+k}$, $Q_{k}$ to local observable event sets, for $i\in I$. We prove that $\bigcup_{i\in I} K_i$ is conditionally normal with respect to those generators and natural projections.

    i) $P_k(\bigcup_{i\in I} K_i)$ is normal with respect to $L(G_k)$ and $Q_{k}$ because $Q_k^{-1}Q_kP_k(\overline{\bigcup_{i\in I} K_i}) \cap L(G_k) = \bigcup_{i\in I} (Q_k^{-1}Q_kP_k(\overline{K_i}) \cap L(G_k)) = \bigcup_{i\in I} P_k(\overline{K_i}) = P_k(\overline{\bigcup_{i\in I}K_i}) = P_k(\bigcup_{i\in I}\overline{K_i})$, where the second equality is by normality of $P_k(K_i)$ with respect to $L(G_k)$ and $Q_{k}$, for $i\in I$.

    ii) Note that $Q_{1+k}^{-1}Q_{1+k}P_{1+k}(\overline{\cup_{i\in I} K_i}) \cap L(G_1) \| P_k(\overline{\cup_{i\in I} K_i}) = \cup_{i\in I} (Q_{1+k}^{-1}Q_{1+k}P_{1+k}(\overline{K_i})) \cap \cup_{i\in I} (L(G_1)\| P_k(\overline{K_i})) = \cup_{i\in I} \cup_{j\in I} (Q_{1+k}^{-1}Q_{1+k}P_{1+k}(\overline{K_i}) \cap L(G_1)\| P_k(\overline{K_j}))$ and $P_{1+k}(\overline{\cup_{i\in I} K_i})\subseteq Q_{1+k}^{-1}Q_{1+k}P_{1+k}(\overline{\cup_{i\in I} K_i}) \cap L(G_1) \| P_k(\overline{\cup_{i\in I} K_i})$. For the sake of contradiction, assume that there exist indexes $i\neq j$ in $I$ such that $Q_{1+k}^{-1}Q_{1+k}P_{1+k}(\overline{K_i}) \cap L(G_1)\|P_k(\overline{K_j})\not\subseteq P_{1+k}(\overline{\cup_{i\in I} K_i})$. Then the left-hand side must be nonempty, which implies that there exists $x\in Q_{1+k}^{-1}Q_{1+k}P_{1+k}(\overline{K_i})\cap L(G_1)\|P_k(\overline{K_j})$ and $x \notin P_{1+k}(\overline{\cup_{i\in I} K_i})$. As $x\in Q_{1+k}^{-1}Q_{1+k}P_{1+k}(\overline{K_i})$, there exists $w\in \overline{K_i}$ such that $Q_{1+k}(x)=Q_{1+k}P_{1+k}(w)$. Applying the projection $P_k':\Sigma_{1+k,o}^*\to \Sigma_{k,o}^*$, we get that $P_k'Q_{1+k}(x)=P_k'Q_{1+k}P_{1+k}(w)$. As $Q_kP_k^{1+k}=P_k'Q_{1+k}$ and $Q_kP_k=P_k'Q_{1+k}P_{1+k}$ (see Fig.~\ref{fig0}), 
    \begin{figure}
      \centering
      \includegraphics[scale=1]{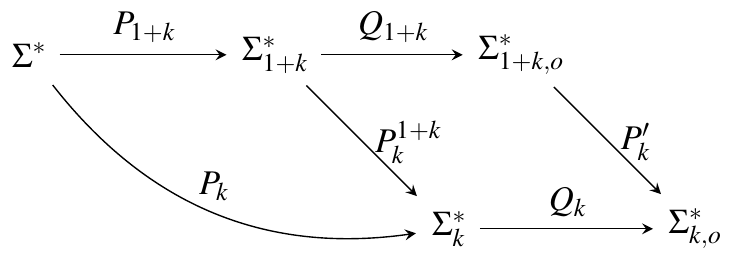}
      \caption{A commutative diagram of the natural projections.}
      \label{fig0}
    \end{figure}
    we have $Q_kP_k^{1+k}(x)=Q_kP_k(w)$, that is, $P_k^{1+k}(x)\in Q_{k}^{-1}Q_{k}P_{k}(\overline{K_i})$. Since $P_k^{1+k}(x)\in P_k(\overline{K_j})\subseteq L(G_k)$, the normality of $P_k(K_i)$ with respect to $L(G_k)$ and $Q_k$ gives that $P_k^{1+k}(x)\in P_{k}(\overline{K_i})$. But then $x\in L(G_1)\|P_k(\overline{K_i})$, and normality of $P_{1+k}(K_i)$ implies that $x\in P_{1+k}(\overline{K_i})\subseteq P_{1+k}(\overline{\cup_{i\in I} K_i})$, which is a contradiction.

    iii) As the last item of the definition is proven in the same way, the theorem holds.
  \end{proof}

  Given generators $G_1$, $G_2$, and $G_k$, let
  \[
    \supccn(K,L,(\Sigma_{1,u},\Sigma_{2,u},\Sigma_{k,u}),(Q_{1+k},Q_{2+k},Q_{k}))
  \]
  denote the supremal conditionally controllable and conditionally normal sublanguage of the specification language $K$ with respect to the plant language $L=L(G_1\| G_2\| G_k)$, the sets of uncontrollable events $\Sigma_{1,u}$, $\Sigma_{2,u}$, $\Sigma_{k,u}$, and projections $Q_{1+k}$, $Q_{2+k}$, $Q_{k}$, where $Q_i:\Sigma_i^* \to \Sigma^*_{i,o}$, for $i=1+k,2+k,k$. 
  
  In the sequel, the computation of the supremal conditionally controllable and conditionally normal sublanguage is investigated. In the same way as in~\cite{SCL}, the following notation is adopted.
  
  Consider the setting of Problem~\ref{problem:controlsynthesis} and define the languages as shown in Fig.~\ref{figCN},
  \begin{figure*}
    \begin{equation}\label{eqCN}
      \boxed{
      \begin{aligned}
        \supcn_k     & = \supcn(P_k(K), L(G_k), \Sigma_{k,u},Q_k)\\
        \supcn_{1+k} & = \supcn(P_{1+k}(K), L(G_1) \| \overline{\supcn_k}, \Sigma_{1+k,u},Q_{1+k})\\
        \supcn_{2+k} & = \supcn(P_{2+k}(K), L(G_2) \| \overline{\supcn_k}, \Sigma_{2+k,u},Q_{2+k})
      \end{aligned}}
    \end{equation}
    \caption{Definition of supremal controllable and normal sublanguages.}
    \label{figCN}
  \end{figure*}
  where $\supcn(K,L,\Sigma_u,Q)$ denotes the supremal controllable and normal sublanguage of $K$ with respect to $L$, $\Sigma_u$, and $Q$. We recall that the supremal controllable and normal sublanguage always exists and equals the union of all controllable and normal sublanguages of $K$, cf.~\cite{CL08}.
  
  \begin{theorem}[\cite{SCL}]\label{thm2}
    Consider the setting of Problem~\ref{problem:controlsynthesis} with a prefix-closed specification $K$ and the languages defined in~(\ref{eqCN}). Let $P^{i+k}_k$ be an $(P^{i+k}_i)^{-1}L(G_i)$-observer and OCC (resp. LCC) for $(P^{i+k}_i)^{-1}L(G_i)$, for $i=1,2$. Assume that the language $P^{1+k}_k(\supcn_{1+k})\cap P^{2+k}_k(\supcn_{2+k})$ is normal with respect to $L(G_k)$ and $Q_k$. Then 
    \begin{multline*}
      \supcn_{1+k} \parallel \supcn_{2+k} \\ = \supccn(K, L, (\Sigma_{1,u}, \Sigma_{2,u}, \Sigma_{k,u}), (Q_{1+k},Q_{2+k},Q_{k}))\,, 
    \end{multline*}
    where $L=L(G_1\|G_2\|G_k)$.
    \hfill\QED
  \end{theorem}
  
  We can now further improve the above result as follows.
  \begin{theorem}\label{thm2b}
    Consider the setting of Problem~\ref{problem:controlsynthesis} and the languages defined in~(\ref{eqCN}). Assume that $\supcn_{1+k}$ and $\supcn_{2+k}$ are nonconflicting and that $P^{1+k}_k(\supcn_{1+k})\cap P^{2+k}_k(\supcn_{2+k})$ is controllable and normal with respect to $L(G_k)$, $\Sigma_{k,u}$, and $Q_k$. Then 
    \begin{multline*}
      \supcn_{1+k} \parallel \supcn_{2+k} \\ = \supccn(K, L, (\Sigma_{1,u}, \Sigma_{2,u}, \Sigma_{k,u}), (Q_{1+k},Q_{2+k},Q_{k}))\,, 
    \end{multline*}
    where $L=L(G_1\|G_2\|G_k)$.
  \end{theorem}
  \begin{proof}
    Let $M = \supcn_{1+k} \parallel \supcn_{2+k}$ and $\supccn = \supccn(K, L, (E_{1+k,u}, E_{2+k,u}, E_{k,u}), (Q_{1+k},Q_{2+k},Q_{k}))$. 
    
    To prove $M\subseteq \supccn$, we show that $M \subseteq P_{1+k}(K)\parallel P_{2+k}(K) = K$ (by conditional decomposability) is conditionally controllable with respect to $L$ and $\Sigma_{1,u}, \Sigma_{2,u}, \Sigma_{k,u}$ (which follows from Theorem~\ref{thmNEW}), and conditionally normal with respect to $L$ and $Q_{1+k},Q_{2+k},Q_{k}$ (which needs to be shown). However, $P_k(M) = P^{1+k}_k(\supcn_{1+k}) \cap P^{2+k}_k(\supcn_{2+k})$ is normal with respect to $L(G_k)$ and $Q_k$ by the assumption. Furthermore, $P_{1+k}(M) = \supcn_{1+k} \parallel P^{2+k}_k(\supcn_{2+k})$. Since $P_{1+k}(M)\subseteq \supcn_{1+k}$ and $P_{k}(M)\subseteq \supcn_{k}$ (by the assumption), $x\in Q_{1+k}^{-1}Q_{1+k}(\overline{P_{1+k}(M)}) \cap L(G_1) \parallel \overline{P_k(M)} \subseteq Q_{1+k}^{-1}Q_{1+k}(\overline{\supcn_{1+k}}) \cap L(G_1) \parallel \overline{\supcn_k} = \overline{\supcn_{1+k}}$ (by normality of $\supcn_{1+k}$). In addition, $P^{1+k}_k(x)\in \overline{P_k(M)} \subseteq P^{2+k}_k(\overline{\supcn_{2+k}})$. Thus, $x\in \overline{\supcn_{1+k}} \parallel P^{2+k}_k(\overline{\supcn_{2+k}}) = \overline{P_{1+k}(M)}$ by the nonconflictingness of the supervisors. The case for $P_{2+k}(M)$ is analogous, hence $M\subseteq \supccn$.

    To prove $\supccn \subseteq M$, it is sufficient by Lemma~\ref{lem11} to show that $P_{i+k}(\supccn)\subseteq \supcn_{i+k}$, for $i=1,2$. To do this, note that $P_{1+k}(\supccn) \subseteq P_{1+k}(K)$ is controllable and normal with respect to $L(G_1) \parallel \overline{P_k(\supccn)}$, $\Sigma_{1+k,u}$, and $Q_{1+k}$ by definition. Since $P_k(\supccn)$ is controllable and normal with respect to $L(G_k)$, $E_{k,u}$, and $Q_k$, it is also controllable and normal with respect to $\overline{\supcn_k}\subseteq L(G_k)$ because $P_k(\supccn)\subseteq \supcn_k$. As $P_{1+k}(\supccn)$ is controllable with respect to $L(G_1)\parallel \overline{P_k(\supccn)}$, and $L(G_1)\parallel \overline{P_k(\supccn)}$ is controllable with respect to $L(G_1)\parallel \overline{\supcn_k}$ by Lemma~\ref{feng}, the transitivity of controllability (Lemma~\ref{lem_transC}) implies that $P_{1+k}(\supccn)$ is controllable with respect to $L(G_1) \parallel \overline{\supcn_k}$ and $\Sigma_{1+k,u}$. Similarly, as $P_{1+k}(\supccn)$ is normal with respect to $L(G_1)\parallel \overline{P_k(\supccn)}$, and $L(G_1)\parallel \overline{P_k(\supccn)}$ is normal with respect to $L(G_1)\parallel \overline{\supcn_k}$ by Lemma~\ref{normalitaComposition}, transitivity of normality (Lemma~\ref{lem_trans}) implies that $P_{1+k}(\supccn)$ is normal with respect to $L(G_1) \parallel \overline{\supcn_k}$ and $Q_{1+k}$. Thus, we have shown that $P_{1+k}(\supccn)\subseteq \supcn_{1+k}$. The case of $P_{2+k}(M)$ is analogous, hence $\supccn \subseteq M$ and the proof is complete.
  \end{proof}
  
  Note that the sufficient condition in Theorem~\ref{thm2b} is not practical for verification, although the intersection is only over the coordinator alphabet that is hopefully small. Unlike controllability, normality is not preserved by natural projections under observer and OCC assumptions. This would require results on hierarchical control under partial observations that are not known so far. Therefore, we propose a condition that is (similarly as in the case of complete observations) stronger than the one of Theorem~\ref{thm2b}, but is easy to check and, moreover, is sufficient for a distributed computation of the supremal conditionally controllable and conditionally normal sublanguage even in the case of non-prefix-closed specifications. Namely, we observe that controllability and normality conditions of Theorem~\ref{thm2b} are weaker than to require that $\supCN_k \subseteq  P_k(\supCN_{i+k})$, for $i=1,2$. The intuition behind the condition $\supCN_k \subseteq  P_k(\supCN_{i+k})$, for $i=1,2$, is that local supervisors (given by $\supCN_{i+k}$) do not need to improve the action by the supervisor for the coordinator on the coordinator alphabet. In this case, the intuition is the same as if the three supervisors (the supervisor for the coordinator and local supervisors) would operate on disjoint alphabets (namely $\Sigma_k$, $\Sigma_1\setminus \Sigma_k$ and $\Sigma_2\setminus \Sigma_k$) and it is well known that there is no problem with blocking and maximal permissiveness in this case (nonconflictness and mutual controllability of modular control) are trivially satisfied.
  
  \begin{proposition}\label{weakern}
    Consider the setting of Problem~\ref{problem:controlsynthesis} and the languages defined in~(\ref{eqCN}). If $\supCN_k \subseteq P_k(\supCN_{i+k})$, for $i=1,2$, then $P_k(\supCN_{1+k}) \cap P_k(\supCN_{2+k})$ is controllable and normal with respect to $L(G_k)$, $\Sigma_{k,u}$, and $Q_k$. 
  \end{proposition}
  \begin{proof}
    First of all, we shown that the inclusion $\supCN_k \supseteq P_k(\supCN_{i+k})$, for $i=1,2$ always holds true. From its definition, $P_k(\supCN_{i+k})\subseteq P_k(L(G_i) \| \overline{\supcn_k})\subseteq \overline{\supcn_k}$ and, clearly, $P_k(\supCN_{i+k})\subseteq P_k(K)$ as well. In order to show that $P_k(\supCN_{i+k})\subseteq \supCN_k$, it suffices to show that  $\overline{\supCN_k}\cap P_k(K)\subseteq \supCN_k$. This can be proven by showing that $\overline{\supCN_k}\cap P_k(K)$ is  controllable and normal with respect to $L(G_k)$, $\Sigma_{k,u}$, and $Q_k$.
    
    For controllability, let $s\in \overline{\overline{\supCN_k}\cap P_k(K)}$, $u\in \Sigma_{k,u}$ with $su\in L(G_k)$. Since there exists $t\in \Sigma_{k}^*$ such that $st\in \overline{\supCN_k}\cap P_k(K)\subseteq \overline{\supCN_k}$, we have that $s\in \overline{\supCN_k}$ as well. Since $\supCN_k$ is controllable with respect to $L(G_k)$ and $\Sigma_{k,u}$, $su\in \overline{\supCN_k}\subseteq \overline{P_k(K)}$. Hence, there exists $v\in \Sigma_{k}^*$ such that $suv\in \supCN_k\subseteq P_k(K)$. Altogether, $suv\in \overline{\supCN_k}\cap P_k(K)$, i.e., $su\in \overline{\overline{\supCN_k}\cap P_k(K)}$. 

    For normality, let $s\in \overline{\overline{\supCN_k}\cap P_k(K)}$ and $s'\in L(G_k)$ with $Q_k(s)=Q_k(s')$.  Recall that $s\in \overline{\supCN_k}$ as well. Again, normality of $\supCN_k$ with respect to $L(G_k)$ and $Q_k$ implies that $s'\in \overline{\supCN_k}$. Thus, there exists $v\in \Sigma_{k}^*$ such that $s'v\in \supCN_k\subseteq P_k(K)$. This implies that $s'v\in \overline{\supCN_k}\cap P_k(K)$, i.e., $s'\in \overline{\overline{\supCN_k}\cap P_k(K)}$, which completes the proof of the inclusion $\supCN_k \supseteq P_k(\supCN_{i+k})$, for $i=1,2$. 

    According to the assumption that the other inclusions also hold, we have the equalities $\supCN_k=P_k(\supCN_{i+k})$, for $i=1,2$. Therefore, $P_k(\supCN_{1+k}) \cap P_k(\supCN_{2+k})=\supCN_k$, which is controllable and normal with respect to $L(G_k)$, $\Sigma_{k,u}$, and $Q_k$ by definition of $\supCN_k$.
  \end{proof} 
   
  Now, combining Proposition~\ref{weakern} and Theorem~\ref{thm2b} we obtain the corollary below.
  \begin{cor}
    Consider the setting of Problem~\ref{problem:controlsynthesis} and the languages defined in~(\ref{eqCN}). If $\supCN_k \subseteq P_k(\supCN_{i+k})$, for $i=1,2$, then 
    \begin{multline*}
      \supcn_{1+k} \parallel \supcn_{2+k} \\= \supccn(K, L, (\Sigma_{1,u}, \Sigma_{2,u}, \Sigma_{k,u}), (Q_{1+k},Q_{2+k},Q_{k}))\,, 
    \end{multline*}
    where $L=L(G_1\|G_2\|G_k)$.
  \end{cor}
  \begin{proof}
    Let $\supccn = \supccn(K, L, (\Sigma_{1,u}, \Sigma_{2,u}, \Sigma_{k,u})$, $(Q_{1+k},Q_{2+k},Q_{k}))$ and $M = \supcn_{1+k} \parallel \supcn_{2+k}$.
    To prove that $M$ is a subset of $\supccn$, we show that (i) $M$ is a subset of $K$, (ii) $M$ is conditionally controllable with respect to generators $G_1$, $G_2$, $G_k$ and uncontrollable event sets $\Sigma_{1,u}$, $\Sigma_{2,u}$, $\Sigma_{k,u}$, and (iii) $M$ is conditionally normal with respect to generators $G_1$, $G_2$, $G_k$ and projections $Q_{1+k}$, $Q_{2+k}$, $Q_{k}$. To this aim, notice that $M$ is a subset of $P_{1+k}(K) \parallel P_{2+k}(K) = K$, because $K$ is conditionally decomposable. Moreover, by Lemma~\ref{lemma:Wonham} and the fact shown in the proof of Proposition~\ref{weakern} that $\supcn_k \supseteq P_k(\supcn_{i+k})$, for $i=1,2$, the language $P_k(M) = P_k(\supcn_{1+k}) \cap P_k(\supcn_{2+k})=\supcn_k$ is controllable and normal with respect to $L(G_k)$, $\Sigma_{k,u}$, and $Q_k$. Similarly, $P_{i+k}(M) = \supcn_{i+k} \parallel P_k(\supcn_{j+k}) = \supcn_{i+k} \parallel \supcn_{k} = \supcn_{i+k}$, for $j\neq i$, which is controllable and normal with respect to $L(G_i)\parallel \overline{P_k(M)}$. Hence, $M$ is a subset of $\supccn$.

    The opposite inclusion is shown in Theorem~\ref{thm2b}, because nonconflictingness is not needed in this direction of the proof.
  \end{proof}
%   \begin{proof}
%     Using Proposition~\ref{weakern} we have that $P_k(\supCN_{1+k}) \cap P_k(\supCN_{2+k})$ is controllable and normal with respect to $L(G_k)$, $\Sigma_{k,u}$, and $Q_k$. It remains to show that $\supcn_{1+k}$ and $\supcn_{2+k}$ are nonconflicting. Then the claim follows directly from Theorem~\ref{thm2b}. We recall that $\Sigma_k = \Sigma_{1+k}\cap \Sigma_{2+k}$, hence the shared event set of $\supcn_{1+k}$ and $\supcn_{2+k}$ is $\Sigma_k$. Consequently, $P_{2+k}(\supCN_{1+k})=P_k(\supCN_{1+k})$ and $P_{1+k}(\supCN_{2+k})=P_k(\supCN_{2+k})$, hence $P_{2+k}(\supCN_{1+k})=\supCN_k =P_{1+k}(\supCN_{2+k})$. By Lemma~\ref{soucin1}, $\supcn_{1+k}$ and $\supcn_{2+k}$ are nonconflicting.
%   \end{proof} 
%    Finally, we will revisit the problem when   $\supccn$ coincides with the monolithic optimal solution $\supcn$.

%%%%%%%%%%%%%%%%%%%%%%%%%%%%%%%%%%%%%%%%%%%%%%%%%%%%%%%%%%%%%%%%%%%%%%%%%%%%%%%%
\section{Conclusion}\label{conclusion}
  In this paper, we have further generalized several results of coordination control of concurrent automata with both complete and partial observations. We have presented weaker sufficient conditions for the computation of supremal conditionally controllable sublanguages and supremal conditionally controllable and conditionally normal sublanguages with simplified concepts of conditional observability and conditional normality. Since our results admit quite a straightforward extension to a multi-level coordination control framework, in a future work we would apply our framework to DES models of engineering systems.

%%%%%%%%%%%%%%%%%%%%%%%%%%%%%%%%%%%%%%%%%%%%%%%%%%%%%%%%%%%%%%%%%%%%%%%%%%%%%%%%
\section{Acknowledgments}
%   The authors gratefully acknowledge comments and suggestions of the anonymous referees.
  This research was supported by the M\v{S}MT grant LH13012 (MUSIC) and by RVO: 67985840.
  
\bibliographystyle{IEEEtranS}
\bibliography{cdc2014}

\newpage
\appendix
  In this section, we list the auxiliary results.
  \begin{lem}[Proposition~4.6 in \cite{FLT}]\label{feng}
    Let $L_i\subseteq \Sigma_i^*$, for $i=1,2$, be prefix-closed languages, and let $K_i\subseteq L_i$ be controllable with respect to $L_i$ and $\Sigma_{i,u}$. Let $\Sigma=\Sigma_1\cup \Sigma_2$. If $K_1$ and $K_2$ are synchronously nonconflicting, then $K_1\parallel K_2$ is controllable with respect to $L_1\parallel L_2$ and $\Sigma_u$. \hfill\QED
  \end{lem}

  \begin{lem}[\cite{automatica2011}]\label{lem_transC}
    Let $K\subseteq L \subseteq M$ be languages over $\Sigma$ such that $K$ is controllable with respect to $\overline{L}$ and $\Sigma_u$, and $L$ is controllable with respect to $\overline{M}$ and $\Sigma_u$. Then $K$ is controllable with respect to $\overline{M}$ and $\Sigma_u$. \hfill\QED
  \end{lem}

  \begin{lem}[\cite{Won04}]\label{lemma:Wonham}
    Let $P_k : \Sigma^*\to \Sigma_k^*$ be a projection, and let $L_i \subseteq \Sigma_i^*$, where $\Sigma_i\subseteq \Sigma$, for $i=1,2$, and $\Sigma_1\cap \Sigma_2 \subseteq \Sigma_k$. Then $P_k(L_1\parallel L_2)=P_k(L_1) \parallel P_k(L_2)$. \hfill\QED
  \end{lem}

  \begin{lem}[\cite{automatica2011}]\label{lem11}
    Let $L_i \subseteq \Sigma_i^*$, for $i=1,2$, and let $P_i : (\Sigma_1\cup \Sigma_2)^* \to \Sigma_i^*$ be a projection. Let $A\subseteq (\Sigma_1\cup \Sigma_2)^*$ such that $P_1(A)\subseteq L_1$ and $P_2(A)\subseteq L_2$. Then $A \subseteq L_1\parallel L_2$. \hfill\QED
  \end{lem}

  \begin{lem}[\cite{pcl06}]\label{obsComposition}
    Let $L_i \subseteq \Sigma_i^*$, for $i\in J$, be languages, and let $\cup_{k,\ell\in J}^{k\neq\ell} (\Sigma_k\cap \Sigma_\ell)\subseteq \Sigma_0 \subseteq (\cup_{i\in J} \Sigma_i)^*$. If $P_{i,0}:\Sigma_i^* \to (\Sigma_i\cap \Sigma_0)^*$ is an $L_i$-observer, for $i\in J$, then $P_{0}:(\cup_{i\in J} \Sigma_i)^* \to \Sigma_0^*$ is an $(\parallel_{i\in J} L_i)$-observer. \hfill\QED
  \end{lem}
  
  \begin{lem}\label{lem_trans}
    Let $K\subseteq L\subseteq M$ be languages such that $K$ is normal with respect to $L$ and $Q$, and $L$ is normal with respect to $M$ and $Q$. Then, $K$ is normal with respect to $M$ and $Q$.
  \end{lem}
  \begin{proof}
    $Q^{-1}Q(\overline{K})\cap \overline{L} = \overline{K}$ and $Q^{-1}Q(\overline{L})\cap \overline{M} = \overline{L}$, hence $Q^{-1}Q(\overline{K})\cap \overline{M} \subseteq Q^{-1}Q(\overline{L})\cap \overline{M} = \overline{L}$. It implies $Q^{-1}Q(\overline{K})\cap \overline{M} = Q^{-1}Q(\overline{K})\cap \overline{M} \cap \overline{L} = \overline{K} \cap \overline{M} = \overline{K}$.
  \end{proof}

  \begin{lem}\label{normalitaComposition}
    Let $K_1\subseteq L_1$ over $\Sigma_1$ and $K_2\subseteq L_2$ over $\Sigma_2$  be nonconflicting languages such that $K_1$ is normal with respect to $L_1$ and $Q_1:\Sigma_1^*\to \Sigma_{1,o}^*$ and $K_2$ is normal with respect to $L_2$ and $Q_2:\Sigma_2^*\to \Sigma_{2,o}^*$. Then $K_1\| K_2$ is normal with respect to $L_1\| L_2$ and $Q:(\Sigma_1\cup\Sigma_2)^*\to (\Sigma_{1,o}\cup\Sigma_{2,o})^*$.
  \end{lem}
  \begin{proof}
    $Q^{-1}Q(\overline{K_1\parallel K_2}) \cap L_1\parallel L_2 \subseteq Q_1^{-1}Q_1(\overline{K_1}) \parallel Q_2^{-1}Q_2(\overline{K_2}) \parallel L_1\parallel L_2 = \overline{K_1} \parallel \overline{K_2} = \overline{K_1\parallel K_2}$. As the other inclusion always holds, the proof is complete.
  \end{proof}

\end{document}